\date{27 August, 2016}
\documentclass[11pt]{amsart}
\usepackage[utf8]{inputenc}
\usepackage[T1]{fontenc}
\usepackage{amsmath}
\usepackage{amsfonts}
\usepackage{amssymb}
\usepackage{amsthm}
\usepackage[all,cmtip]{xy}
\usepackage{mathrsfs}
\usepackage{mathabx}
\usepackage{color}
\definecolor{cobalt}{RGB}{61,89,171}
\usepackage[colorlinks,citecolor=cobalt,linkcolor=cobalt,urlcolor=cobalt,pdfpagemode=UseNone,backref = page]{hyperref}

\def\bC{\mathbb{C}}

\def\bQ{\mathbb{Q}}
\def\bR{\mathbb{R}}
\newcommand{\R}{\mathbb{R}}
\def\bS{\mathbb{S}}
\def\bT{\mathbb{T}}
\def\bZ{\mathbb{Z}}
\def\bK{\mathbb{K}}

\def\bk{\mathbf{k}}

\def\cL{\mathcal{L}}
\def\cM{\mathcal{M}}

\def\rk{\textrm{rk}}

\def\o{\omega}

\theoremstyle{plain}
\newtheorem{theorem}{Theorem}[section]
\newtheorem{proposition}[theorem]{Proposition}
\newtheorem{lemma}[theorem]{Lemma}
\newtheorem{corollary}[theorem]{Corollary}
\newtheorem{conjecture}[theorem]{Conjecture}

\theoremstyle{definition}
\newtheorem{definition}[theorem]{Definition}
\newtheorem{example}[theorem]{Example}
\newtheorem{examples}[theorem]{Examples}
\newtheorem{remark}[theorem]{Remark}

\newcommand{\secref}[1]{Section~\ref{#1}}
\newcommand{\thmref}[1]{Theorem~\ref{#1}}
\newcommand{\propref}[1]{Proposition~\ref{#1}}
\newcommand{\lemref}[1]{Lemma~\ref{#1}}

\newcommand{\conjref}[1]{Conjecture~\ref{#1}}

\newcommand{\defref}[1]{Definition~\ref{#1}}

\begin{document}

\title{Hereditary Properties of Co-K\"ahler Manifolds}

\author[G. Bazzoni]{Giovanni Bazzoni}
\address{Fakult\"at f\"ur Mathematik\\
Universit\"at Bielefeld\\
33501 Bielefeld, Germany}
\email{gbazzoni@gmail.com}

\author[G. Lupton]{Gregory Lupton}
\address{Department of Mathematics\\
Cleveland State University\\
Cleveland OH \\
44115  USA}
\email{g.lupton@csuohio.edu}

\author[J. Oprea]{John Oprea}
\address{Department of Mathematics\\
Cleveland State University\\
Cleveland OH \\
44115  USA}
\email{j.oprea@csuohio.edu}

\date{\today}


\begin{abstract} We show how certain topological properties of co-K{\"a}hler manifolds
derive from those of the K\"ahler manifolds which construct them. We go beyond
Betti number results and describe the cohomology algebra structure of co-K\"ahler
manifolds. As a consequence, we prove that co-K\"ahler manifolds satisfy the
Toral Rank Conjecture: ${\rm dim}(H^*(M;\bQ)) \geq 2^r$, for any $r$-torus $T^r$
which acts almost freely on $M$. 
\end{abstract}

\thanks{This work was partially supported by grants from the Simons Foundation (\#209575 to
Gregory Lupton and \#244393 to John Oprea) and an INdAM (Istituto Nazionale di Alta Matematica)
fellowship (to Giovanni Bazzoni).}

\maketitle
\dedicatory{\hfill \emph{In memory of Sergio Console}}

\section{Introduction}\label{sec:intro}
Co-K\"ahler manifolds may be thought of as odd-dimensional versions of K\"ahler manifolds and various structure theorems explicitly display how the former are constructed
from the latter (see \cite{BO,Li}).

In this paper, we take the point of view that topological and geometric properties of co-K\"ahler manifolds are inherited from those of the K\"ahler manifolds that construct them. We shall see this in 
both topological and geometric contexts. First, let us recall some basic definitions (see \cite{Bl} for a detailed introduction).

\begin{definition}\label{def:cosymp}
An \textbf{almost contact metric structure} $(J,\xi,\eta,g)$ on a manifold $M^{2n+1}$ consists of a tensor $J$ of type $(1,1)$, a vector field $\xi$, a 1-form $\eta$ and a Riemannian metric $g$ such 
that
\begin{equation}\label{eq:1}
J^2 = -I + \eta \otimes \xi,\quad \eta(\xi)=1, \quad g(JX,JY)=g(X,Y)-\eta(X)\eta(Y),
\end{equation}
for vector fields $X$ and $Y$, $I$ the identity transformation on $TM$.
\end{definition}

A local $J$-basis for $TM$, $\{X_1,\ldots,X_n,JX_1,\ldots,JX_n,\xi\}$,
may be found with $\eta(X_i)=0$ for $i=1,\ldots,n$. The \emph{fundamental
$2$-form} on $M$ is given by
\[\omega(X,Y) = g(JX,Y),\]
and if $\{\alpha_1,\ldots,\alpha_n,\beta_1,\ldots,\beta_n,\eta\}$ is
a local $1$-form basis dual to the local $J$-basis, then
\[\omega = \sum_{i=1}^n \alpha_i \wedge \beta_i.\]
Note that $\imath_\xi \omega = 0$.

\begin{definition}
The geometric structure $(M^{2n+1},J,\xi,\eta,g)$ is 
\begin{itemize}
\item \textbf{co-symplectic} if $d\omega=0=d\eta$;
\item \textbf{normal} if $[J,J]+2\, d\eta \otimes \xi=0$;
\item \textbf{co-K\"ahler} if it is co-symplectic and normal; equivalently, if $J$ is parallel with respect to the metric $g$.
\end{itemize}
%
%a
%{\bf co-K\"ahler} structure on $M$ if
%$$[J,J]+2\, d\eta \otimes \xi=0 \ \, {\rm and}\ \, d\omega=0=d\eta$$
%or, equivalently, $J$ is parallel with respect to the metric $g$.
\end{definition}

Recently, co-symplectic geometry has attracted a great deal of interest, especially in the context of Poisson geometry, where co-symplectic structures are interpreted as corank 1 Poisson structures 
(see for instance \cite{CdNY1,CoMa,FMM,GMP,MT}). Sasakian structures also belong to this family; more precisely, they are normal structures such that $d\eta=\omega$ (see \cite{BoGa,CdNY2,CdNMY}).

Two crucial facts about co-K\"ahler manifolds are contained in the following lemma.
For a direct proof of these facts, see \cite{BO}.

\begin{lemma}\label{lem:parallel}
On a co-K\"ahler manifold, the vector field $\xi$ is Killing and parallel.
Furthermore, the 1-form $\eta$ is parallel and harmonic.
\end{lemma}

\noindent \lemref{lem:parallel} is a key point in \thmref{thm:cosympsplit}
below. In fact, in \cite{Li} it is shown that we can replace $\eta$ by a
harmonic integral form $\eta_\theta$ with dual parallel vector field
$\xi_\theta$ and associated metric $g_\theta$, $(1,1)$-tensor
$J_\theta$ and closed $2$-form $\omega_\theta$ with $i_{\xi_\theta}\omega_\theta
=0$. Then we have the following result of H. Li.

\begin{theorem}[\cite{Li}] \label{thm:maptor}
If $M^{2n+1}$ is compact, with the structure $(J_\theta,\xi_\theta,\eta_\theta,g_\theta)$,
there are a compact K\"ahler manifold $(K,h)$ and a Hermitian isometry
$\psi\colon K \to K$ such that $M$ is diffeomorphic to the mapping torus
\[K_\psi = \frac{K \times [0,1]}{(x,0)\sim (\psi(x),1)}\]
with associated fibre bundle $K \to M=K_\psi \to S^1$.
\end{theorem}

In \cite{BO}, the following refinement of Li's result is proved:

\begin{theorem}[\cite{BO}, Theorem 3.3]\label{thm:cosympsplit}
Let $(M^{2n+1},J,\xi,\eta,g)$ be a compact co-K\"ahler manifold with integral structure
and mapping torus bundle $K \to M \to S^1$. Then $M$ splits as
$M \cong S^1 \times_{\mathbb Z_m} K$,
where $S^1 \times K \to M$ is a finite cover with structure group
$\mathbb Z_m$ acting diagonally and by translations on the
first factor. Moreover, $M$ fibres over the circle $S^1/(\mathbb Z_m)$
with finite structure group.
\end{theorem}

The first true study of the topological properties of co-K\"ahler manifolds was
made in \cite{CdLM} where the focus was on things such as Betti numbers and
a modified Lefschetz property. The two results above allow us to say something
about the fundamental group and, moreover, to display the higher homotopy groups
as those of the constituent K\"ahler manifold $K$ (groups which, of course, are
generally unknown as well). Nevertheless, the principle (which gives rise to the paper's
title) remains that topological qualities of a co-K\"ahler manifold are intimately tied up with
those of the K\"ahler manifold that constructs it. In this paper, we shall
explore this principle in several ways. We begin by examining the cohomology algebra
of a co-K\"ahler manifold and its effect on the manifold's rational homotopy structure.
We then will consider the structure of the minimal models (in the sense of Sullivan)
of co-K\"ahler manifolds in terms of the decompositions given in Theorem~\ref{thm:maptor}
and \thmref{thm:cosympsplit}. In \secref{sec:toralrank}, we go beyond algebraic
considerations in showing that co-K\"ahler manifolds satisfy the so-called
Toral Rank Conjecture. This theorem strongly connects the geometry of the
co-K\"ahler manifold to the size of its cohomology. 

In a previous version of this paper, a fourth section was included, which has now been taken out and constitutes the new paper \cite{BLO2}.

We have written this paper for an audience of geometers who may not be experts
in rational homotopy. Therefore, we have included a substantial review of basic
facts in the subject. The main references for this material
are \cite{FHT, FHT2, FOT}.

\section{The Lefschetz Property and Associated Algebraic Models}
\subsection{Cohomology Algebra Structure}\label{subsec:cohom}
Using \thmref{thm:cosympsplit}, the following description of the cohomology of a compact
co-K\"ahler manifold was obtained in \cite{BO}.

\begin{theorem}[\cite{BO}, Theorem 4.3]\label{thm:betti}
If $(M^{2n+1},J,\xi,\eta,g)$ is a compact co-K\"ahler manifold with integral structure
and splitting $M \cong K \times_{\mathbb Z_m} S^1$, then
\begin{equation}\label{G_invariant}
H^*(M;\bR) \cong H^*(K;\bR)^G \otimes H^*(S^1;\bR),
\end{equation}
as commutative graded algebras, where $G = \bZ_m$. Hence, the Betti numbers of $M$ satisfy:
\vskip5pt

\indent (i)\quad $b_s(M) = \overline b_s(K) + \overline b_{s-1}(K)$, where
$\overline b_s(K)$ denotes the dimension of $G$-invariant cohomology $H^s(K;\bR)^G$;
\vskip7pt
\indent (ii)\quad $b_1(M) \leq b_2(M) \leq \ldots \leq b_n(M) = b_{n+1}(M)$.
\end{theorem}

In order to study cohomological properties of co-K\"ahler manifolds, we recall the notion of
cohomologically K\"ahlerian differential graded algebra.

\begin{definition}
Let $(A,d)$ be a commutative differential graded algebra (cdga) of cohomological dimension $2n$,
whose cohomology algebra satisfies Poincar\'e duality. The cdga $(A,d)$ is called \emph{cohomologically
K\"ahlerian} if there exists a closed element $\o\in A^2$ such that the map
\[
\cL^{n-p}\colon H^p(A)\to H^{2n-p}(A), \quad [\sigma]\mapsto [\o]^{n-p}\cdot[\sigma]
\]
is an isomorphism for every $0\leq p\leq n$. Note that we include the case where $(A,d)$ has $d=0$ and
we then refer to $A$ as a commutative graded algebra (cga).
\end{definition}
Clearly, if $K$ is a K\"ahler manifold, the cohomology algebra $H^*(K;\mathbf{k})$ is cohomologically K\"ahler,
where $\mathbf{k}$ can be $\bQ$, $\bR$ or $\bC$. Note that there exist examples of non-K\"ahler
manifolds whose de Rham algebra is cohomologically K\"ahler (see for instance \cite{FMS}).

Let $(M,J,\xi,\eta,g)$ be a compact co-K\"ahler manifold with integral structure and mapping
torus bundle $K\to M\to S^1$ and consider the cohomology algebra $H^*(K;\bR)$ of the K\"ahler manifold
$K$. The finite group $G\cong\bZ_m$ acts on $H^*(K;\bR)$ and, according to \thmref{thm:betti}, the
cohomology algebra of $M$ is the product of the invariant part of the cohomology of $K$ and the cohomology of
$S^1$. We now show that the invariant part of $H^*(K;\bR)$ is a cohomologically K\"ahlerian algebra.

\begin{proposition}\label{prop:cohomolKaehler}
The cga $H^*(K;\bR)^G$ is cohomologically K\"ahlerian.
\end{proposition}
\begin{proof}
In \cite[Lemma 4.2]{BO}, it is proved that $H^*(K;\bR)^G$ contains a $G$-invariant element $\o$ of
degree 2 which behaves like a symplectic form. Such an element is the pullback of the K\"ahler form
in $K\times S^1$ under the inclusion $K\hookrightarrow K\times S^1$. In order to see that $H^*(K;\bR)^G$
is cohomologically K\"ahler, we must show, further, that the Lefschetz map $H^p(K;\bR)^G\to H^{2n-p}(K;\bR)^G$
is an isomorphism for every $0\leq p\leq n$. We check injectivity first.
Since $H^*(K;\bR)$ is cohomologically K\"ahlerian, the multiplication by $\o^{n-p}$ is injective on the
whole space $H^p(K;\bR)$ and remains injective when restricted to $H^p(K;\bR)^G$. For surjectivity, take $\tau\in H^{2n-p}(K;\bR)^G$. Again, since $H^*(K;\bR)$ is cohomologically K\"ahlerian, there exists $\tau'\in H^p(K;\bR)$
such that $\tau=\tau'\wedge\o^{n-p}$. We must show that $\tau'\in H^p(K;\bR)^G$. We have
\[
\tau'\wedge\o^{n-p}=\tau=g(\tau)=g(\tau'\wedge\o^{n-p})=g(\tau')\wedge g(\o^{n-p})=g(\tau')\wedge \o^{n-p},
\]
so $(\tau'-g(\tau'))\wedge\o^{n-p}=0$. But $H^*(K;\bR)$ is cohomologically K\"ahlerian, so the multiplication
by $\o^{n-p}$ is injective, and therefore $g(\tau')=\tau'$ and $\tau'\in H^p(K;\bR)^G$.
\end{proof}
This immediately gives another old result about the Betti numbers of co-K\"ahler manifolds
\cite[Theorem 11]{CdLM}.

\begin{corollary}
Let $(M^{2n+1},J,\xi,\eta,g)$ be a compact co-K\"ahler manifold. Then, for $0\leq i\leq n$, the differences $b_{2i+1}(M)-b_{2i}(M)$ are even integers and non-negative if $0\leq i\leq\lfloor\frac{n}{2}\rfloor$.
\end{corollary}
\begin{proof}
By \thmref{thm:betti}, once we replace the co-K\"ahler structure on $M$ with an integral one, we obtain $b_i(M)=\bar{b}_i(K)+\bar{b}_{i-1}(K)$, where $M$ sits in the mapping torus fibration $K\to M\to S^1$,
with $K$ a $2n$-dimensional K\"ahler manifold and $\bar{b}_i(K)=\dim H^i(K;\bR)^G$. We proved in
\propref{prop:cohomolKaehler} that $H^*(K;\bR)^G$ is cohomologically K\"ahler, and this implies that
$\bar{b}_{2i+1}(K)$ is even for $0\leq i\leq n-1$. Therefore,
\begin{align*}
b_{2i+1}(M)-b_{2i}(M)&=\bar{b}_{2i+1}(K)+\bar{b}_{2i}(K)-\bar{b}_{2i}(K)-\bar{b}_{2i-1}(K)=\\
&=\bar{b}_{2i+1}(K)-\bar{b}_{2i-1}(K),
\end{align*}
which is an even number and non-negative if $0\leq i\leq\lfloor\frac{n}{2}\rfloor$.
\end{proof}

Now, according to \thmref{thm:maptor}, the choice of an integral co-K\"ahler structure
$(J_\theta,\xi_\theta,\eta_\theta,g_\theta)$ on $M$ produces a mapping torus bundle $K\to M\to S^1$.
By \thmref{thm:cosympsplit}, this gives in turn a finite cover $K\times S^1\to M$ with deck group
$G\cong\bZ_m$. We then have the following \emph{fundamental data} for a co-K\"ahler manifold.
\begin{definition}
Let $(M,J,\xi,\eta,g)$ be a compact co-K\"ahler manifold and choose an integral structure $(J_\theta,\xi_\theta,\eta_\theta,g_\theta)$ on $M$. The data $(K,G)$ of the mapping torus
bundle $K\to M\to S^1$ and the finite $G$-cover $K\times S^1\to M$ form a {\bf presentation}
of $(M,J,\xi,\eta,g)$.
\end{definition}

Clearly the presentation of $(M,J,\xi,\eta,g)$ depends on the choice of an integral co-K\"ahler
structure on $M$. Nevertheless, according to \thmref{thm:betti},
$H^*(M;\bQ)\cong H^*(K;\bQ)^G\otimes H^*(S^1;\bQ)$ and therefore, if we are given two
different presentations
\[
K_1\to M\to S^1, \quad K_1\times S^1\stackrel{/G_1}{\to} M 
\]
and
\[
K_2\to M\to S^1, \quad K_2\times S^1\stackrel{/G_2}{\to} M,
\]
then we must have
\[
H^*(K_1;\bQ)^{G_1}\cong H^*(K_2;\bQ)^{G_2}\,.
\]

Let $(M,J,\xi,\eta,g)$ be a compact co-K\"ahler manifold and fix a presentation $(K,G)$. The finite cyclic
group $G$ acts on the product $K\times S^1$ and, according to \thmref{thm:cosympsplit}, $G$ acts by
translations on $S^1$, so the $G$ action is free. But the action need not remain free when we restrict
it to $K$. Therefore, the quotient space $K/G$ need not be a manifold. However, since $G$ is a finite,
cyclic group, $K/G$ is a K\"ahler orbifold which is canonically associated to the presentation of the
co-K\"ahler manifold $M$. Indeed, $G$ acts by Hermitian isometries on $K$, so the K\"ahler structure
is preserved under the $G$-action, and passes to the quotient. Since $G$ is finite, we have
$H^*(K/G;\bQ)\cong H^*(K;\bQ)^G$, so the rational cohomology of the quotient $K/G$ is computed
by the invariant rational cohomology of $K$.
In view of (\ref{G_invariant}), the cohomology of $M$ contains information about the cohomology of the
K\"ahler orbifold $K/G$. Such a K\"ahler orbifold $K/G$ is associated to the chosen presentation
$K\to M\to S^1$ of the co-K\"ahler manifold $(M,J,\xi,\eta,g)$, but since all possible presentations
yield diffeomorphic $M$, the corresponding orbifolds have the same rational cohomology.

\subsection{Rational Homotopy Structure}\label{subsec:rht}
In fact, the algebra splitting of \thmref{thm:betti} tells us much more about the structure of the
co-K\"ahler manifold $M$. For this we need to recall some notions from Rational Homotopy Theory.
The reader is referred to \cite{FHT}, \cite[Chapters 2 and 3]{FOT} and \cite{BG} for details and proofs of
the statements that follow.

A commutative graded algebra (cga) over a field of characteristic zero $\bk$, $A$, is called
\emph{free graded commutative} if $A$ is the quotient of $TV$, the tensor algebra on the graded vector space $V$,
by the bilateral ideal generated by the elements $a\otimes b - (-1)^{\vert a\vert \cdot \vert b\vert}
b\otimes a$, where $a$ and $b$ are homogeneous elements of $A$. As an algebra, $A$ is the tensor
product of the symmetric algebra on $V^{\mbox{\scriptsize even}}$ with the exterior algebra on
$V^{\mbox{\scriptsize odd}}$:
\[
A = \mbox{Symmetric}(V^{\mbox{\scriptsize even}})\otimes
\mbox{Exterior}(V^{\mbox{\scriptsize odd}})\,. 
\]

\noindent We denote the free commutative graded algebra on the graded vector space $V$ by $\land V$.
Note that this notation refers to a free commutative graded algebra and not necessarily to an exterior
algebra alone. We usually write $\land V = \land (x_i)$, where $x_i$ is a homogeneous basis of $V$.
Clearly the cohomology of a cdga is a commutative graded algebra. A morphism of cdga's inducing
an isomorphism in cohomology will be called a \emph{quasi-isomorphism}.
A \emph{Sullivan cdga} is a cdga $(\land V,d)$ whose underlying
algebra is free commutative, with $V = \{\, V^n\,\}$, $n\geq 1$, and such that $V$ admits a basis
$x_\alpha$ indexed by a well-ordered set such that $d(x_\alpha) \in\land(x_\beta)_{\beta<\alpha}$.
A \emph{(Sullivan) minimal cdga} is a Sullivan cdga $(\land V,d)$ satisfying the additional property that
$d(V)\subset \land^{\geq 2} V$. Minimal cdga's play an important role because they are tractable models for
``all'' other cdga's. (For the path-connected non-simply-connected case of the following result, see
\cite[Chapter 6]{Hal} or, from a functorial viewpoint, \cite{BG}, especially Chapters 7 and 12.)

\begin{theorem}[Existence and Uniqueness of the Minimal Model]
\label{thm:existuniqminmod} \hfill\newline
Let $(A,d)$ be a cdga over $\bk$ satisfying $H^0(A,d) = \bk$,
where $\bk$ is $\mathbb R$ or $\mathbb Q$ and ${\rm dim}(H^p(A,d))<\infty$
for all $p$. Then,
\begin{enumerate}
\item There is a quasi-isomorphism $\varphi \colon (\land V,d) \to (A,d)$,
where $(\land V,d)$ is a minimal cdga.
\item  The minimal cdga $(\land V,d)$ is unique in the following sense:
If $(\land W,d)$ is a minimal cdga and $\psi\colon (\land W,d) \to
(A,d)$ is a quasi-isomorphism, then there is an isomorphism
$f\colon (\land  V,d) \to (\land W,d)$ such that $\psi\circ f$ is
homotopic (see \cite{FHT}) to $\varphi$.
\end{enumerate}

\noindent The cdga $(\land V,d)$ is then called the \emph{minimal
model} of $(A,d)$.
\end{theorem}

The connection between this type of algebra and topology is via the de Rham cdga of differential forms on the
manifold $M$, $(\Omega(M),d)$, when $\bk$ is $\mathbb R$ and Sullivan's rational polynomial forms on $M$
(thought of as a simplicial complex, say), $(A_{PL}(M),d)$, when $\bk$ is $\mathbb Q$. Applying
\thmref{thm:existuniqminmod} to these cdga's produces a \emph{minimal model of the space} $M$ denoted by
$\varphi\colon \mathcal M_M = (\land V,d) \to A$, where we let $A$ stand for either the de Rham or Sullivan
algebras. We shall not distinguish the minimal models depending on the field because the context will
always be clear. The minimal model thus provides a special type of cdga associated to a space. Note that
the condition $H^0(A,d) = \bk$ in \thmref{thm:existuniqminmod} means that any path-connected space
has a minimal model. There are two key facts that make minimal cdga's an important tool.

\begin{lemma}\label{lem:liftinglem}\hfill\newline
\vspace{-10pt}
\begin{enumerate}
\item If $f \colon (\land V,d) \to (\land Z,d)$ is a quasi-isomorphism
between minimal cdga's, then $f$ is an isomorphism.
\item For a Sullivan cdga $(\land V,d)$, a cdga quasi-isomorphism $f \colon (A,d) \to (B,d)$
and a cdga morphism $\varphi \colon (\land V,d) \to (B,d)$, there is a cdga morphism
$\psi \colon (\land V,d) \to (A,d)$ such that $f\circ \psi$
is homotopic (see \cite{FHT}) to $\varphi$.
\[
\xymatrix{
& (A,d)\ar[d]^f\\
(\land V,d)\ar[ur]^{\psi}\ar[r]_{\varphi} & (B,d)}
\]
\end{enumerate}
\end{lemma}

Here is one application. Say that the spaces $X$ and $Y$ have \emph{the same rational homotopy type}
if there is a finite chain of maps $X \to Y_1 \leftarrow Y_2 \to \cdots \to Y$ such that each induced
map in rational cohomology is an isomorphism. If we consider the cdga morphisms
$$\mathcal M_{Y_1} \to A_{PL}(Y_1) \to A_{PL}(X) \leftarrow \mathcal M_X$$
and apply (2) of \lemref{lem:liftinglem}, we obtain a cdga morphism $\mathcal M_{Y_1} \to \mathcal M_X$
which is a quasi-isomorphism (since the other morphisms are). By (1) of \lemref{lem:liftinglem},
we then have $\mathcal M_{Y_1} \cong \mathcal M_X$. We carry on this process through the chain of maps
to get $\mathcal M_{Y} \cong \mathcal M_X$.

\begin{proposition}\label{prop:rhtype}
If $X$ and $Y$ have the same rational homotopy type, then their minimal models are isomorphic.
Moreover, if $X$ and $Y$ are nilpotent spaces (e.g. simply connected), then the converse is true.
\end{proposition}
The second statement follows from the existence of spatial rationalizations coming from homotopical
localization theory. In general, these do not exist for non-nilpotent spaces. This is important to
note because compact co-K\"ahler manifolds are rarely nilpotent spaces (they are never simply connected of
course). So, in the case of non-nilpotent spaces such as typical co-K\"ahler manifolds, it
is the isomorphism class of the minimal model that really represents some sort of rational type.
Of course, everything we have said applies to models over $\mathbb R$ as well.

Some minimal models are even more special; they are isomorphic to the minimal models associated to
the cohomology algebra (considered as a cdga with zero differential). Spaces with this property
are called \emph{formal}. \lemref{lem:liftinglem} implies that there is the following equivalent
definition.

\begin{definition}\label{def:formal}
A space $X$, with minimal model $(\land V,d)$, is called \textbf{formal} if there is a quasi-isomorphism
$$\theta \colon (\land V,d) \to (H^*(X;\bQ),0)\,.$$
\end{definition}

\begin{remark}\label{rem:formal}
We can also define a cdga $(A,d)$ to be \emph{formal}
if there is a chain of quasi-isomorphisms
$$(A,d) \leftarrow (B_1,d_1) \rightarrow\cdots (B_k,d_k) \rightarrow (H^*(A),0)\,.$$
We can take the minimal models of $(A,d)$, the minimal models  of the $(B_i,d_i)$ and the minimal
models of the morphisms and apply \lemref{lem:liftinglem} to see that this is equivalent to
\defref{def:formal}.
\end{remark}

The last piece of Rational Homotopy Theory that we shall need is the notion of an equivariant
minimal model. Let $\Gamma$ be a finite group. A $\Gamma$-cdga is a cdga
on which the group $\Gamma$ acts by a homomorphism $\Gamma \to \mbox{aut}_{cdga} (A,d_A)$.

\begin{definition}\label{def:gminimalmodel}
A $\Gamma$-cdga $(A,d_A)$ is called $\Gamma$-\textbf{minimal} if $(A,d_A) = (\land V,d)$ with
\begin{enumerate}
\item $d(V) \subset \land^{\geq 2} (V)$;
\item Each $V^n$ is a $\Gamma$-module (i.e. this gives a
$\Gamma$-structure to $\land V$);
\item $d$ is $\Gamma$-equivariant: $d(ga)=gd(a)$;
\item $V$ admits a filtration by sub $\Gamma$-spaces
\[0 \subset V(0) \subset V(1) \subset\cdots \subset V(n) \subset
\cdots \subset V = \cup_n V(n)\,,\]
with $d(V(n)) \subset (\land V(n-1))$.
\end{enumerate}
\end{definition}
\noindent Generalizing the non-equivariant case, we have the following. Note that,
while all proofs (e.g. \cite[Theorem 3.26]{FOT}) of this result assume $H^1(A,d_A)=0$,
this is for convenience only. In the same way that the ordinary minimal model
can be constructed for general path-connected spaces (see \thmref{thm:existuniqminmod})
by a limiting process, we can also construct an equivariant model.

\begin{theorem}\label{thm:gminimal}
Let $(A,d_A)$ be a $\Gamma$-cdga. Suppose that $H^0(A,d_A) = \bk$, where
$\bk=\R$, or $\bk=\bQ$.
Then there exists a $\Gamma$-minimal algebra
$(\land V,d)$ and a $\Gamma$-equivariant quasi-isomorphism
$\varphi \colon (\land V,d) \to (A,d_A)$. The $\Gamma$-minimal
algebra $(\land V,d)$ is called the $\Gamma$-minimal model of the
$\Gamma$-cdga $(A,d_A)$, and it is unique up to $\Gamma$-isomorphism.
\end{theorem}

Suppose a finite group $\Gamma$ acts on the space $X$. If the
$\Gamma$-equivariant minimal model of $X$, $(\land V,d)$, is
equivariantly isomorphic to the $\Gamma$-equivariant minimal
model of $H^*(X;\bk)$, then we say that $(X,\Gamma)$ is
$\Gamma$-\emph{formal}. It can be shown that a formal $\Gamma$-space
is $\Gamma$-formal \cite{Pap}. That is, if $X$ is a formal space with an action
of a finite group $\Gamma$, then the equivariant minimal model can be constructed
from the action of $\Gamma$ on $H^*(X;\bk)$. Moreover, in this situation,
we can show that the minimal model of $X/\Gamma$ is the minimal
model of $H^*(X/\Gamma;\bk)$, so that $X/\Gamma$ is formal. To see
this, let $\phi\colon (\land W,d) \to (\land V,d)^\Gamma$ be the
minimal model of $(\land V,d)^\Gamma$. By computing the invariant part of cohomology,
we know that $(\land W,d)$ is the minimal model of $X/\Gamma$ (see \cite[Corollary 3.29]{FOT}).
Now consider the commutative diagram below, where the right square comes
from the inclusion of invariant elements and the equivariant formality
quasi-isomorphism $\theta$, and the left square
comes from lifting the composition $(\land W,d) \stackrel{\phi}{\to}
(\land V,d)^\Gamma \stackrel{\theta^\Gamma}{\to} H^*(X;\bk)^\Gamma$
through the isomorphism $H^*(X/\Gamma;\bk) \cong H^*(X;\bk)^\Gamma$.

\[\xymatrix{
(\land W,d) \ar[r]^-\phi \ar[d] & (\land V,d)^\Gamma \ar[r]
\ar[d]^-{\theta^\Gamma} &
(\land V,d) \ar[d]^-\theta \\
H^*(X/\Gamma;\bk) \ar[r]^-\cong & H^*(X;\bk)^\Gamma \ar[r] & H^*(X;\bk)     }
\]
Then, since $\theta$ is a quasi-isomorphism, so is $\theta^\Gamma$. But then the lift
$(\land W,d) \to H^*(X/\Gamma;\bk)$ is also a quasi-isomorphism. Hence,
$X/\Gamma$ is formal if $X$ is.

Let $(M,J,\xi,\eta,g)$ be a compact co-K\"ahler manifold and let $(K,G)$ be a presentation. Let
$(\cM_K,d)$ denote the minimal model of $K$. Then its invariant part is a rational model for the space $K/G$.
A main result of \cite{DGMS} states that compact K\"ahler manifolds are formal. This means, among other
things, that the minimal model of a K\"ahler manifold $K$ is determined by its rational
cohomology algebra $H^*(K;\bQ)$. (Also, note that formality does not depend on the field
$\bk$.) Since $K$ is a formal space, so is $K/G$, and hence its minimal
model can be computed from the cohomology algebra $H^*(K/G;\bQ)\cong H^*(K;\bQ)^G$. Furthermore,
co-K\"ahler manifolds are also formal (see \cite{CdLM}), so the rational minimal model of $M$ can
be constructed directly from its rational cohomology, which in view of \thmref{thm:betti}, is isomorphic to $H^*(K;\bQ)^G\otimes H^*(S^1;\bQ)$, for any presentation $(K,G)$ of $M$.
Putting this together, we obtain the following result.

\begin{theorem}\label{minimal_orbifold}
Let $M$ be a compact co-K\"ahler manifold and let $(K,G)$ be a presentation. Then the minimal
model of $M$ has the following cdga splitting:
\[\cM_M \cong \cM_{K/G} \otimes \cM_{S^1}.\]
\end{theorem}

\begin{proof}
Because all spaces are formal, we have the following diagram:
\[\xymatrix{
\cM_M \ar[d]_-{\theta_M} \ar@{-->}[r]^-\phi & \cM_{K/G} \otimes \cM_{S^1} \ar[d]^-{\theta_{K/G} \otimes \theta_{S^1}} \\
H^*(M;\bQ) \ar[r]^-\cong & H^*(K;\bQ)^G \otimes H^*(S^1;\bQ),
}
\]
where the top arrow comes from \lemref{lem:liftinglem}. Also by \lemref{lem:liftinglem}, we see that
$\phi$ is an isomorphism.
\end{proof}

\thmref{minimal_orbifold} is quite interesting, in the following sense. Since a compact 
co-K\"ahler manifold $(M,J,\xi,\eta,g)$ is never simply connected, when its fundamental group 
$\pi_1(M)$ is not nilpotent, or acts non-nilpotently on higher homotopy groups (see \cite{FOT}), 
the minimal model of $M$ does not give, in general, information about the usual rational homotopy 
structure of $M$. For instance, we don't see things such as rational homotopy groups and Whitehead products. 
However, the isomorphism class of a minimal model is always an invariant attached to any space,
so \thmref{minimal_orbifold} says that inside the minimal model of $M$ we can see the 
``auxiliary'' K\"ahler orbifold $K/G$ (i.e. its minimal model). So the minimal model provides
a new type of (geometric) information that is non-classical.

\begin{example}
Here is another description of the Chinea-de Le\'on-Marrero example contained in \cite{CdLM}.
Consider the torus $T^2$ with its standard K\"ahler structure and let $\phi\colon T^2\to T^2$ be the
holomorphic isometry covered by the linear transformation $A\colon \bR^2\to\bR^2$,
\[
A=\begin{pmatrix}
0 & -1\\
1 & 0
\end{pmatrix}.
\]
The Betti numbers of the mapping torus $T^2_\phi$ are easily computed to be the following:
\begin{itemize}
\item $b_0(T^2_\phi)=b_3(T^2_\phi)=1$;
\item $b_1(T^2_\phi)=1$, generated by the volume form of the circle $S^1$;
\item $b_2(T^2_\phi)=1$, generated by the K\"ahler class of the torus $T^2$.
\end{itemize}
The minimal model of $T^2_\phi$ is
\[
(\wedge(t,u,v),|t|=1, |u|=2, |v|=3, dv=u^2),
\]
which is isomorphic to the minimal model of $S^2\times S^1$. The automorphism $\phi$ of $T^2$ has
order 4 and $M$ can be seen as the quotient of $T^2\times S^1$ by the $\bZ_4$-action given by
\[
(x,y,z)\mapsto (y,-x,z+1/4).
\]
Now consider the quotient $T^2/G$. When we think of $T^2$ as the square $[0,1]\times[0,1]$ with the
sides identified, the action of $G$ on $T^2$ is a rotation of $\pi/2$ around the center of the square.
There are therefore 2 fixed points, $(0,0)$ and $(\frac{1}{2},\frac{1}{2})$. Using the Riemann-Hurwitz
formula, one sees that the quotient $T^2/G$ is a compact surface of genus 0, hence topologically a
sphere $S^2$.
\end{example}

\begin{example}
Take two copies $(\bC P^1_i,\o_i)$, $i=1,2$, of $\bC P^1$ with its standard K\"ahler structure, and
consider the manifold $K=\bC P^1\times\bC P^1$ endowed with K\"ahler structure $\o=\o_1+\o_2$.
Let $\phi\colon K\to K$ denote the map $\phi(p,q)=(q,p)$. Then $\phi$ is a holomorphic isometry
of $K$. The rational cohomology of the co-K\"ahler manifold $N=K_\phi$ is:
\begin{itemize}
\item $H^1(K_\phi;\bQ)=\langle [u]\rangle$, generated by the class of the circle $S^1$;
\item $H^2(K_\phi;\bQ)=\langle [\o]\rangle$, generated by the K\"ahler class of $K$;
\item $H^3(K_\phi;\bQ)=\langle [\o\wedge u]\rangle$;
\item $H^4(K_\phi;\bQ)=\langle [\o^2]\rangle$;
\item $H^5(K_\phi;\bQ)=\langle [\o^2\wedge u]\rangle$.
\end{itemize}
\end{example}
The minimal model of $K_\phi$ is
\[
(\wedge(t,u,v),|t|=1, |u|=2, |v|=5, dv=u^3),
\]
which is isomorphic to the minimal model of $\bC P^2\times S^1$. The automorphism $\phi$ of $K$ has
order 2 and $M$ can be seen as the quotient of $K\times S^1$ by the $\bZ_2$-action given by
\[
(p,q,t)\mapsto (q,p,t+1/2).
\]
It is not hard to see that the quotient $K/\bZ_2$ is smooth, and isomorphic (as algebraic varieties)
to $\bC P^2$. Indeed, let $D=(p,p)\subset K$ be the diagonal; the Segre map gives an embedding
$\imath\colon K\to\bC P^3$ which realizes $K$ as a smooth quadric $\mathcal{Q}$.
The projection from $\mathcal{Q}$ to a plane $\pi\subset \bC P^3$ is a $2:1$ cover, branched over the
conic $\mathcal{C}\subset \pi$ which is the image under the projection of $\imath(D)$. Therefore, the quotient $\mathcal{Q}/\bZ^2$ is precisely $\pi\cong\bC P^2$.

\begin{remark}
It is worth pointing out here that in neither of the two examples does the minimal model compute typical
rational homotopy information (beyond cohomology) about the corresponding K\"ahler mapping torus.
Indeed, in the first case, $T^2_\phi$ is an aspherical manifold, as can be seen directly from the long exact
sequence of homotopy groups of the fibration $T^2\to T^2_\phi\to S^1$, but the minimal model of $T^2_\phi$ has
generators in degree 2 and 3. In the second case, by the same method one sees that $\pi_2(K_\phi)=\bZ\oplus\bZ$,
but the minimal model of $K_\phi$ has only one generator in degree 2. In both cases, the reason for this 
apparent mis-match is that neither $T^2_\phi$ nor $K_\phi$ are nilpotent spaces.
\end{remark}

\section{Toral Rank of Co-K\"ahler Manifolds}\label{sec:toralrank}
The \emph{Toral Rank Conjecture} (TRC), due to Halperin \cite{Hal85}, has been a very influential and
motivating problem in the development of Rational Homotopy Theory. In this section we show that
a co-K{\"a}hler manifold satisfies the conjecture.  This is again an instance of our principle that
co-K\"ahler manifolds inherit properties from their constituent K\"ahler manifolds since the TRC
has long been known in the K\"ahler case (see \cite{AP,LO}). Before we state the conjecture, recall that
a compact Lie group $G$ (continuously) acts \emph{almost freely} on a space $X$ if all isotropy groups
are finite. The \emph{toral rank} of a space $X$, $\rk(X)$, is the dimension of the largest torus that can act
almost freely on $X$.

\begin{conjecture}[Toral Rank Conjecture]\label{conj: TRC}
If the toral rank of a space $X$ is $r$, then
\[\mathrm{dim}\  H^*(X;\bQ)  \geq 2^r.\]
\end{conjecture}

The notation $\dim\,V$ means (total) dimension of $V$ as a rational graded vector space. Our methods
allow us to establish this conjecture for a large class of spaces, which (strictly) contains
co-K{\"a}hlerian manifolds. Furthermore we obtain a strong form of the conjecture;  namely, we will
show that, for our class of spaces, the rational cohomology algebra actually contains a
``cohomological $r$-torus''. Note that toral rank is a homeomorphism invariant, but is not a homotopy
invariant. This suggests that we are getting at deeper topological qualities of co-K\"ahler
manifolds than Betti numbers or even the full algebra structure of cohomology.
We begin with some terminology.

\begin{definition}[Property B]
Say that a graded algebra $H$ has \textbf{Property B} if, for any negative-degree derivation
$\theta$ of $H$, we have
\[
\theta\left( H^1 \right) = 0 \implies \theta( H) = 0. 
\]
We say that a space $X$ has Property B if its (rational) cohomology algebra has Property B.
\end{definition}

For example, any simply connected space whose rational cohomology algebra does not admit a non-zero,
negative-degree derivation has Property B.  Also, it is known that any cohomologically K{\"ahlerian
space has Property B. This fact is due to Blanchard \cite[Th.II.1.2]{Blan}, and this accounts for
our choice of the letter B here.  Of course, since the property is intrinsic to the cohomology algebra,
any space with the same cohomology algebra has Property B. A main result about Property B spaces is
the following (see for instance \cite[Proposition 4.40, Theorem 4.36]{FOT}).

\begin{proposition}\label{prop:BSS}
Suppose $F\to E \to X$ is a fibration such that $F$ satisfies Property B and $X$ is simply connected.
In the Leray-Serre spectral sequence, if $d_2(H^1(F;\bQ))=0$, then the spectral sequence collapses
and $H^*(E;\bQ) \cong H^*(X;\bQ) \otimes H^*(F;\bQ)$ as $H^*(X;\bQ)$-modules. In particular, if
$F$ is cohomologically K\"ahlerian and $d_2(H^1(F;\bQ))=0$, then the spectral
sequence collapses.
\end{proposition}

Now suppose that we have an action $T^r \times X \to X$, of an $r$-torus on a space $X$.  Recall that we
say the action is \emph{homologically injective} if the orbit map $T^r \to X$ of the action induces an
injection $H_1(T^r;\bQ) \to H_1(X;\bQ)$ on first rational homology groups. This property of actions
has been extensively studied (see, e.g. \cite{CoRa71}). In \cite{AP}, Allday and Puppe show that a
cohomologically K{\"a}hlerian space satisfies  \conjref{conj: TRC}. (In \cite{LO}, this can be extended to
spaces of Lefschetz type.) Distilling their argument a little reveals that it is really Property B
that is the key, and not the cohomologically K{\"a}hlerian structure, as such. In the following result,
we extend the Allday-Puppe result by relaxing their hypothesis.  Nonetheless, the basic argument,
which we repeat here for the convenience of the reader, remains that of \cite[Th.2.2]{AP}.

\begin{theorem}\label{thm: B TRC}
Let $X$ be a space that satisfies Property B above. If an $r$-torus $T^r$ acts almost freely on
$X$, then the action is homologically injective, and we have
$$H^*(X;\bQ) \cong H^*(X_{T^r};\bQ) \otimes H^*(T^r;\bQ)$$
as graded algebras, where $X_{T^r}=ET^r \times_{T^r} X$ is the Borel construction.
In particular, we have
\[
\mathrm{dim}\  H^*(X;\bQ) \geq  \mathrm{dim}\  H^*(T^r;\bQ)  = 2^r\,, 
\]
and thus $X$ satisfies  \conjref{conj: TRC}.
\end{theorem}

\begin{proof}
Suppose that a torus $\bT=T^{r}$ acts almost freely on the space $X$ for some $r$.
Let $E\bT\to B\bT$ be a universal principal $\bT-$bundle and let $X_{\bT}=(X\times E\bT)/\bT$ be the
Borel construction. Let $\{E^{p,q}_k\}$ be the rational cohomology Leray-Serre spectral sequence of
$X \to X_{\bT}\to B\bT$ and let $s$ be the rank of the linear map
\[
d_2\colon E^{0,1}_2=H^1(X;\bQ)\to E^{2,0}_2=H^2(B\bT;\bQ)\,.
\]
Now, we can choose a basis for $H^*(B\bT;\bQ)$ so that $H^*(B\bT;\bQ)\cong\bQ[a_1,\ldots,a_r]$
with $|a_i|=2$ for $i=1,\ldots,r$ and $d_2(y_i)=a_i$, $i=1,\ldots,s$ for $y_1,\ldots,y_s\in
H^1(X;\bQ)$. Since $d_2$ is a derivation, we obtain
\[
d_2(y_{i_1}\cdots y_{i_{j+1}})=\sum_{\ell=1}^{j+1}\pm a_{i_{\ell}}\otimes y_{i_1}\cdots \hat{y}_{i_{\ell}}\cdots y_{i_{j+1}}.
\]
By induction, using the algebraic independence of the $a_j$, we see that $y\coloneq y_1\cdots y_s$ must
also be non-zero.

Suppose that $s < r$.
By duality, the Hurewicz theorem and the fact that elements of $\pi_1(\bT)$ are realizable by
homomorphisms from $S^1$, we can obtain a sub-torus $\bS \subseteq \bT$ which realizes the
subalgebra $\langle a_1,\ldots,a_s \rangle$. Now, every sub-torus of a torus has a complement, so let $\bK\subseteq
\bT$ be such that $\bT = \bS \times \bK$. In particular, $\dim(\bK)=r-s$.
We then see that $\bK$ is the sub-torus such that the ideal
generated by the $a_i$, $i=1,\ldots,s$ is the kernel of the projection in cohomology:
\[(a_1,\ldots,a_s)=\ker(H^*(B\bT;\bQ)\to H^*(B\bK;\bQ))\,.\]
We now restrict the action of $\bT$ on $X$ to $\bK$ and note that it is also almost free.
If we form the Borel fibration for the $\bK$ action, then the Leray-Serre spectral sequence for
$X_\bT$ pulls back to that for $X_\bK$. Then, because ${\rm Im}((d_2)_\bT) \subseteq
\ker(H^*(B\bT;\bQ)\to H^*(B\bK;\bQ))$, we have $(d_2)_\bK=0$ on $H^1(X;\bQ)$. But because $X$ satisfies
Property B, \propref{prop:BSS} guarantees that the spectral sequence collapses. However, this implies
that $H^*(B\bK;\bQ) \to H^*(X_\bK;\bQ)$ is injective and the Borel fixed point theorem then
says that the fixed point set $X^\bK$ is non-empty, contradicting the fact that $\bK$ acts almost
freely. Hence, $\bK$ is trivial and $r=s$.

Thus we have $y_1,\ldots,y_r \in H^*(X;\bQ)$ which
generate an exterior algebra. In fact, stepping back in the Barratt-Puppe sequence to the
fibration $\bT \to X \to X_\bT$, we see that $\langle y_1,\ldots,y_r \rangle$ maps onto $H^*(\bT;\bQ)$. Therefore
this spectral sequence collapses and $H^*(X;\bQ) \cong H^*(X_\bT;\bQ) \otimes H^*(\bT;\bQ)$.
Thus, ${\rm dim}(H^*(X;\bQ)) \geq {\rm dim}(H^*(\bT;\bQ)) = 2^r$.
\end{proof}

Next, we show that the class of graded algebras that satisfy Property B is closed under tensor products.

\begin{proposition}\label{prop: tensor B}
If $H$ and $G$ are graded algebras that satisfy Property B, then so too $H \otimes G$ satisfies Property B.
\end{proposition}

\begin{proof}
Suppose $H$ and $G$ have Property B, and that $\theta \colon H \otimes G \to H\otimes G$  is a
negative degree derivation that vanishes on $(H \otimes G)^1 = H^1\otimes 1 + 1 \otimes G^1$.
We wish to show that $\theta$ must be zero.

First, we show that $\theta$ vanishes on H. For suppose that $\theta(H) \neq= 0$, and let
$k \geq 0$ be the smallest integer for which $\theta(H) \cap H\otimes G^k \neq= 0$.
Take any $\chi \in H$, and write $\theta(\chi) = \theta_k(\chi) + \theta_{k+1}(\chi)$, with
$\theta_k(\chi) \in H \otimes G^k$ and $\theta_{k+1}(\chi) \in I(G^{\geq k+1})$, the ideal of
$H \otimes G$ generated by elements of $G$ of degree $k+1$ or greater.  Further, suppose that
we have a basis $\{g^i\}$ of $G^k$.  Then we may write $\theta_k(\chi) = \sum_i \theta^i_k(\chi)\otimes g^i$.
This defines linear maps $\theta^i_k \colon H \to H$, of  negative degree --- in fact of degree equal
to $|\theta| - k$.  It is straightforward to check that each $\theta^i_k$ is a derivation of $H$, so
$\theta^i_k(\chi) = 0$, for each $\chi \in H^1$, by the assumption that $H$ has Property B.
But this implies that we have $\theta(H) \subseteq  H\otimes G^{\geq k+1}$, which contradicts
our assumption on $k$.  Therefore, we must have $\theta(H) = 0$. The same argument, with $H$ and
$G$ interchanged, gives that $\theta$ must vanish on G. Hence, $\theta = 0$ and  $H \otimes G$ has Property B.
\end{proof}

\begin{corollary}\label{cor:cokahlerB}
If $M$ is a compact co-K\"ahler manifold, then it satisfies the Toral Rank Conjecture.
\end{corollary}

\begin{proof}
By \thmref{thm:betti}, $H^*(M;\bR) = H^*(K;\bR)^G \otimes H^*(S^1;\bR)$ and by
\propref{prop:cohomolKaehler}, $H^*(K;\bR)^G$ is K\"ahlerian and has Property B.
Clearly $H^*(S^1;\bR)$ has Property B for degree reasons. Hence, by \propref{prop: tensor B},
$H^*(M;\bR)$ has Property B. Now apply \thmref{thm: B TRC}.
\end{proof}

This result points out again that properties of co-K\"ahler manifolds often derive from
properties of the constituent K\"ahler maifold. Also note that, by \cite{BO}, a
co-K\"ahler manifold always has toral rank at least equal to one. Note that we also have the
following result, where the $1$ is added to account for the $S^1$ factor in cohomology.

\begin{corollary}\label{torank}
Let $(M,J,\xi,\eta,g)$ be a compact co-K{\"a}hler manifold with presentation $(K,G)$
so that $M=(K\times S^1)/G$. Then
\[
\rk(M)\leq\tilde{\alpha}_1(K)+1\,,
\]
where $\tilde{\alpha}_1(K)$ is the maximal number of algebraically independent elements in
$H^1(K;\bQ)$ which are fixed by the induced $G$-action on $H^1(K;\bQ)$.
\end{corollary}

By the Myers-Steenrod theorem \cite{MS}, the isometry group $\mathrm{Isom}(M,g)$ of a compact
Riemannian manifold is a compact Lie group. As a consequence, it is observed in \cite{BO} that,
when $(M,J,\xi,\eta,g)$ is compact co-K{\"a}hler, the closure of the Reeb flow in $\mathrm{Isom}(M,g)$
is a compact torus $T$, which acts almost freely on $M$. Therefore $M$ is endowed with an almost free torus action.
\begin{corollary}\label{Welshtorus}
Let $(M,J,\xi,\eta,g)$ be a compact co-K{\"a}hler manifold and assume that $M=(K\times S^1)/G$ for a
K\"ahler manifold $K$. Let $T\subset\mathrm{Isom}(M,g)$ be the closure of the Reeb flow in the isometry group
of $M$. Then
\[
\dim(T)\leq\tilde{\alpha}_1(K)+1\,.
\]
When $b_1(M)=1$, then $\dim(T)=1$ and the Reeb flow generates a homologically injective circle action on $M$.
\end{corollary}
\begin{proof}
The torus $T$ acts almost freely on $M$, hence $\dim(T)\leq\rk(M)$. By Corollary \ref{torank}, we get $\dim(T)\leq\rk(M)\leq\tilde{\alpha}_1(K)+1$.
When $b_1(M)=1$, we have $H^1(K;\bQ)^G=0$ by \thmref{thm:betti}, so the group $G$ fixes no element
on $H^1(K;\bQ)$. Therefore $\tilde{\alpha}_1(K)=0$
and $\rk(M)\leq 1$. Notice that $\dim(T)\geq 1$, since the flow of the Reeb vector field $\xi$ generates
at least a circle in $\mathrm{Isom}(M,g)$. Therefore we get $1\leq\dim(T)\leq\rk(M)\leq 1$, and $T=S^1$,
hence $\xi$ generates a circle action, which is homologically injective by the argument given in \cite[Section 2]{BO}.
\end{proof}

\begin{examples}\hfill\newline
\vspace{-10pt}
\begin{enumerate}
\item As already observed, any cohomologically K{\"a}hlerian space satisfies Property B.
\item Any algebra generated in degree $1$ (tautologically) satisfies Property B.  In particular,
this remark applies to $H^*(T^r;\bQ)$ for each $r \geq 1$.  Note that these algebras are
cohomologically K{\"a}hlerian only for even $r$.
\item Therefore, if $H$ is cohomologically K{\"a}hlerian and $G$ is
generated in degree $1$, then the tensor product $H \otimes G$ satisfies Property B.
\item Suppose that $H$ is a finite-dimensional complete intersection, i.e., generated by even-degree
generators with ideal of relations generated by a (maximal length) regular sequence. Another
long-standing, open conjecture due to Halperin is that any such algebra does not admit any
non-zero negative-degree derivation.  This conjecture has been established in many cases.  If $H$ is
any such algebra for which this conjecture is true, then tensor products of the form $H \otimes G$,
with $G$ generated in degree $1$ (or, more generally, any algebra with Property B) satisfy Property B.
\end{enumerate}
Note that these examples include many that are neither cohomologically K{\"a}hlerian, nor
finite-dimensional complete intersections.
\end{examples}

\begin{remark}
If $X$ and $Y$ are spaces that satisfy Property B, then by \thmref{thm: B TRC} each satisfies
\conjref{conj: TRC}, and by \propref{prop: tensor B} and once again \thmref{thm: B TRC}, their
product $X \times Y$ also satisfies \conjref{conj: TRC}.  In this way,  we are able to generate
spaces that are products, and that  satisfy \conjref{conj: TRC}.  It is worth  emphasizing that,
in general,  the toral rank---the maximum rank of a torus that may act almost freely --- may not
behave well with respect to products.  In \cite{JeLu04},  an example is given of a product
$X \times Y$ that admits a free circle action, and yet neither $X$ nor $Y$ admit an almost-free
circle action.  Generally, therefore, the toral rank does not behave in a ``sub additive" way with
respect to products. This means, in particular, that as yet there is no \emph{a priori} reason to
conclude $X \times Y$ satisfies \conjref{conj: TRC}, simply because $X$ and $Y$ do.
\end{remark}

\section*{Acknowledgments}
The first author thanks Cleveland State University for its hospitality in July 2013 when most of the results of the present paper were conceived. We would also like to thank the referee for his/her useful comments.

\end{document}